\documentclass[10pt, reqno]{amsart}
\usepackage{tikz-cd}
\usepackage{amssymb,latexsym,amscd,amsmath,amsthm,manfnt,tikz}
\usepackage{mathdots}
\usepackage{xcolor}
\usepackage{hyperref}
\hypersetup{colorlinks=true,linkcolor=blue, citecolor=blue}
\usepackage[widespace]{fourier}
\usepackage{frcursive}
\usepackage[pagewise]{lineno}

\usepackage[matrix,arrow]{xy}
\begin{document}
\input{amssym.def}
\input{xy}

\xyoption{all}
\theoremstyle{remark}
\theoremstyle{plain}
\newtheorem{exam}{Example}
\pagestyle{plain}
\newtheorem*{theorem*}{Theorem}
\setcounter{tocdepth}{1}
\numberwithin{equation}{subsection} 
\newtheorem{guess}{theorem}[section]
\newtheorem{thm}[guess]{Theorem}
\newtheorem{lem}[guess]{Lemma}
\newtheorem{prop}[guess]{Proposition}
\newtheorem{Cor}[guess]{Corollary}
\newtheorem{defi}[guess]{Definition}
\newcommand{\ds}{\displaystyle}
\theoremstyle{definition}
\newtheorem{rem}[guess]{Remark}

\newtheorem{ex}[guess]{Example}
\newtheorem{prob}[guess]{Problem}
\newtheorem{claim}[guess]{Claim}

\newcommand{\Hfr}{\mathfrak{H}}

\newcommand{\fX}{\mathfrak{X}}
\newcommand{\gfr}{\mathfrak{g}}
\newcommand{\hfr}{\mathfrak{h}}
\newcommand{\Dfr}{\mathfrak{D}}
\newcommand{\Efr}{\mathfrak{E}}
\newcommand{\Gfr}{\mathfrak{G}}
\newcommand{\Ffr}{\mathfrak{F}}
\newcommand{\Cfr}{\mathfrak{C}}
\newcommand{\rfr}{\mathfrak{r}}
\newcommand{\lfr}{\mathfrak{l}}
\newcommand{\sfr}{\mathfrak{s}}
\newcommand{\pfr}{\mathfrak{p}}
\newcommand{\Sfr}{\mathfrak{S}}

\newcommand{\cY}{\mathcal{Y}}

\newcommand{\cV}{\mathcal{V}}
\newcommand{\cZ}{\mathcal{Z}}
\newcommand{\cU}{\mathcal{U}}
\newcommand{\cI}{\mathcal{I}}
\newcommand{\cK}{\mathcal{K}}
\newcommand{\cD}{\mathcal{D}}
\newcommand{\cQ}{\mathcal{Q}}
\newcommand{\cO}{\mathcal{O}}
\newcommand{\cC}{\mathcal{C}}
\newcommand{\cE}{\mathcal{E}}
\newcommand{\cG}{\mathcal{G}}
\newcommand{\cW}{\mathcal{W}}
\newcommand{\cB}{\mathcal{B}}
\newcommand{\cF}{\mathcal{F}}
\newcommand{\cH}{\mathcal{H}}
\newcommand{\cM}{\mathcal{M}}
\newcommand{\cA}{\mathcal{A}}
\newcommand{\cN}{\mathcal{N}}
\newcommand{\cR}{\mathcal{R}}
\newcommand{\cL}{\mathcal{L}}
\newcommand{\cT}{\mathcal{T}}
\newcommand{\cS}{\mathcal{S}}
\newcommand{\cX}{\mathcal{X}}
\newcommand{\cP}{\mathcal{P}}
\newcommand{\DD}{\mathcal{D}}
\newcommand{\KK}{\mathcal{K}}
\newcommand{\bs}{\mathbf{s}}
\newcommand{\bb}{\mathbf{b}}
\newcommand{\mf}{\mathfrak{f}}

\newcommand{\Prym}{\mathrm{Prym}}
\newcommand{\Nbd}{\mathrm{Nbd}}
\newcommand{\Frac}{\mathrm{Frac}}
\newcommand{\AbShv}{\mathrm{AbShv}}
\newcommand{\Shv}{\mathrm{Shv}}
\newcommand{\PreShv}{\mathrm{Preshv}}
\newcommand{\Sets}{\mathrm{Sets}}
\newcommand{\topgp}{\mathrm{top}}
\newcommand{\sub}{\mathrm{sub}}
\newcommand{\Sym}{\mathrm{Sym}}
\newcommand{\disc}{\mathrm{disc}}
\newcommand{\Pic}{\mathrm{Pic}}
\newcommand{\Gal}{\mathrm{Gal}}
\newcommand{\Jac}{\mathrm{Jac}}
\newcommand{\Obj}{\mathrm{Obj}}
\newcommand{\Stab}{\mathrm{Stab}}
\newcommand{\Div}{\mathrm{Div}}
\newcommand{\Grass}{\mathrm{Grass}}
\newcommand{\height}{\mathrm{height}}
\newcommand{\Bl}{\mathrm{Bl}}
\newcommand{\Ad}{\mathrm{Ad}}
\newcommand{\Inn}{\mathrm{Inn}}
\newcommand{\Out}{\mathrm{Out}}

\newcommand{\Ram}{\mathrm{Ram}}
\newcommand{\diag}{\mathrm{diag}}
\newcommand{\ram}{\mathrm{ram}}
\newcommand{\Mor}{\mathrm{Mor}}
\newcommand{\Img}{\mathrm{Img}}
\newcommand{\Spec}{\mathrm{Spec}}
\newcommand{\zero}{\mathrm{zero}}
\newcommand{\SU}{\mathrm{SU}}
\newcommand{\Ker}{\mathrm{Ker}}
\newcommand{\Trace}{\mathrm{Trace}}
\newcommand{\ad}{\mathrm{ad}}
\newcommand{\stdn}{\mathrm{stdn}}
\newcommand{\tr}{\mathrm{tr}}
\newcommand{\triv}{\mathrm{triv}}
\newcommand{\sgn}{\mathrm{sgn}}
\newcommand{\ev}{\mathrm{ev}}
\newcommand{\Tor}{\mathrm{Tor}}
\newcommand{\opp}{\mathrm{opp}}
\newcommand{\Cov}{\mathrm{Cov}}
\newcommand{\Covet}{\mathrm{Cov_{\acute{e}t}}}
\newcommand{\ob}{\mathrm{ob}}
\newcommand{\AbGps}{\mathrm{AbGps}}
\newcommand{\rank}{\mathrm{rank}}
\newcommand{\supp}{\mathrm{supp}}
\newcommand{\glue}{\mathrm{\tt glue}}
%\newcommand{\Sets}{\mathrm{Sets}}
% Pfeile

\newcommand{\lra}{\longrightarrow}
\newcommand{\Lra}{\Leftrightarrow}
\newcommand{\Ra}{\Rightarrow}
\newcommand{\hra}{\hookrightarrow}
\newcommand{\sr}{\stackrel}
\newcommand{\dra}{\dashrightarrow}
\newcommand{\ra}{\rightarrow}
\newcommand{\ol}{\overline}
\newcommand{\wh}{\widehat}
\newcommand{\loms}{\longmapsto}
\newcommand{\la}{\leftarrow}
\newcommand{\lems}{\leftmapsto}
\newcommand{\vp}{\varpi}
\newcommand{\ep}{\epsilon}
\newcommand{\La}{\Lambda}
\newcommand{\abf}{\bigtriangleup_\theta}
\newcommand{\ms}{\mapsto}
\newcommand{\evt}{\tilde{\ev}}
\newcommand{\ul}{\underline}
\newcommand{\uT}{\underline{T}}
\newcommand{\cech}{\cC^.(\fU, \uT)}
\newcommand{\bc}{{\mathbb C}}
\newcommand{\tcr}{\text{\cursive r}}
\newcommand{\bp}{{\mathbb P}}
\newcommand{\bz}{{\mathbb Z}}
\newcommand{\bq}{{\mathbb Q}}
\newcommand{\bn}{{\mathbb N}}
\newcommand{\bg}{{\mathbb G}}
\newcommand{\br}{{\mathbb R}}
\newcommand{{\bh}}{{\mathbb H}}
%\newcommand{\<}{\langle}
%\newcommand{\>}{\rangle}

%Greek and other symbols

\newcommand{\vf}{\varphi}

\newcommand{\vt}{\vartheta}

\newcommand{\TT}{\Theta}
\newcommand{\spec}{{\rm Spec}\,}

\newcommand{\gt}{\theta}
\newcommand{\RR}{\mathbb{R}}
\newcommand{\EE}{\mathbb{E}}
\newcommand{\WW}{\mathbb{W}}
\newcommand{\FF}{\mathbb{F}}
\newcommand{\VV}{\mathbb{V}}
\newcommand{\qq}{\mathbb{q}}
\newcommand{\HH}{\mathbb{H}}
\newcommand{\PP}{\mathbb{P}}
\newcommand{\Sbb}{\mathbb{S}}
\newcommand{\ZZ}{\mathbb{Z}}
\newcommand{\GG}{\mathbb{G}}
\newcommand{\QQ}{\mathbb{Q}}
\newcommand{\NN}{\mathbb{N}}
\newcommand{\UU}{\mathbb{U}}
\newcommand{\XX}{\mathbb{X}}
\newcommand{\CC}{\mathbb{C}}
\newcommand{\bX}{\mathbb{X}}
\newcommand{\bY}{\mathbb{Y}}
\newcommand{\fU}{\mathfrak{U}}

\newcommand{\tcm}{\text{\cursive c}}
\newcommand{\tcs}{\small\text{\cursive s}}
\newcommand{\tcq}{\small\text{\cursive q}}
\newcommand{\tcf}{\tiny\text{\cursive f}}
\newcommand{\tcb}{\it b}
\newcommand{\tcge}{\small\text{\cursive ge}}

%\newcommand{\fX}{\mathfrak{X}}
%%%%%%%%%%%%%%%%%%

\newcommand{\sma}[1]{ \begin{pmatrix}

\begin{smallmatrix} #1 \end{smallmatrix} \end{pmatrix} }

\newcommand{\sm}[1]{

\begin{smallmatrix} #1 \end{smallmatrix}  }

\newcommand{\symp}[1]{\mathrm{Sp}_{#1}(\ZZ)}
\renewcommand\qedsymbol{\tt{Q.E.D}}

\newcommand{\Ext}{\mathrm{Ext}}
\newcommand{\std}{\mathrm{std}}
\newcommand{\End}{\mathrm{End}}
\newcommand{\Mum}{\mathrm{Mum}}
\newcommand{\Hom}{\mathrm{Hom}}
\newcommand{\Ind}{\mathrm{Ind}}
\newcommand{\Id}{\mathrm{Id}}
\newcommand{\dimn}{\mathrm{\text{dim}}}
\newcommand{\pardeg}{\mathrm{pardeg}}
\newcommand{\x}{\mathrm{x}}
\newcommand{\Res}{\mathrm{Res}}
\newcommand{\Endq}{\mathrm{End}_\QQ}
\newcommand{\GL}{\mathrm{GL}}
\newcommand{\SL}{\mathrm{SL}}
\newcommand{\Cl}{\mathrm{Cl}}
\newcommand{\SO}{\mathrm{SO}}
\newcommand{\Sp}{\mathrm{Sp}}
\newcommand{\Spin}{\mathrm{Spin}}
\newcommand{\PSO}{\mathrm{PSO}}
\newcommand{\PGL}{\mathrm{PGL}}
\newcommand{\Cb}{\mathrm{Cb}}
\newcommand{\Cg}{\mathrm{Cg}}
\newcommand{\Nm}{\mathrm{Nm}}
\newcommand{\Proj}{\mathrm{Proj}}
\newcommand{\im}{\mathrm{im}}
\newcommand{\Aut}{\mathrm{Aut}}
\newcommand{\coker}{\mathrm{coker}}
\def\map#1{\ \smash{\mathop{\longrightarrow}\limits^{#1}}\ }

%%%%%%%%%%%%%%%%%%%%%%%%%%%%%%%%%%%%%%%%

%short forms for repeatedly used symbols
\newcommand{\Pdon}{\Prym(\pi,\Lambda)}
\newcommand{\twinv}{{H^1(Z,\underline{T})}^W}
\newcommand{\He}{H_{\acute{e}t}}
\newcommand{\Shalf}{S_{1/2}}

\title{On a theorem of Narasimhan and Ramanan on deformations}
%\title{COHOMOLOGY OF LINE BUNDLES ON THE MODULI STACK OF TORSORS UNDER PARAHORIC GROUP SCHEME $\mathcal{G}$}

\author[V. Balaji]{Vikraman Balaji}
\address{Chennai Mathematical Institute }
\email{balaji@cmi.ac.in}

\author[Y. Pandey]{Yashonidhi Pandey}
\thanks{The research for this paper was partially funded by the SERB Core Research Grant CRG/2022/000051.}
\address{ 
Indian Institute of Science Education and Research, Mohali Knowledge city, Sector 81, SAS Nagar, Manauli PO 140306, India}
\email{ ypandey@iisermohali.ac.in, yashonidhipandey@yahoo.co.uk}

\begin{abstract}
Let $X$ be a smooth projective curve genus $G$ (as elaborated in \ref{main1}), over an algebraically closed field $k$ of arbitrary characteristics. Let $\cH$ {\em be a tamely ramified  absolutely simple, simply connected connected group scheme (see \eqref{quasisplitcase})}. Let  $\cM$ denote the moduli stack $\cM_X(\cH)$ of $\cH$-torsors on $X$ and $\cM^{^s}$ be the open substack of {\em stable torsors}. Using the theory of parahoric torsors and Parahoric-correspondences, we describe the cohomology groups $\text{H}^i\left(\cM^{^s}, \cT_{_{\cM}}\right), i = 0,1,2$ and $\text{H}^i\left(\cM^{^s}, \Omega_{_{\cM}}\right),  i = 0,1,2$ in terms of the curve $X$. The classical results of Narasimhan and Ramanan are derived as a consequence. 
%In the last section we deal with the general situation of a quasi-split absolutely simple, simply connected connected group scheme $\cH$. 

\end{abstract}
\maketitle

%{\flushright\epigraph{\tiny{\tt {"For us, there is only the trying. The rest is not our business."― T.S. Eliot,  \\cf.Bhagavad Gita, II.47}}}}
\small
\tableofcontents
\setcounter{tocdepth}{2} % Set depth to 1
\normalsize

\section{Introduction}

Let $k$ be an algebraically closed field of arbitrary characteristics,  and let $G$ be an {\em almost simple,  simply-connected} connected group scheme over $k$. Let $X$ be a smooth projective curve over $k$ of genus $g$ with bounds as elaborated in \ref{main1}. Let $\cM$ denote the moduli stack $\cM_X(G)$ of principal $G$-bundles on $X$, also called $\text{Bun}_X(G)$ in the literature. It is known that $\cM$ is a smooth algebraic $k$-stack. 

Let $E \to X \times \cM$ denote the universal $G$-torsor and let $\gfr := \text{Lie}(G)$. For $t$ in the open substack $\cM^{^s}$ of stable torsors,  it is known that the cohomology group $\text{H}^1 \left(X, E_t \left(\gfr \right) \right)$ controls the deformations of the $G$-torsor $E_t$ and is in fact the tangent space to $\cM^{^s}$ at $E_t$. Let $q:X \times \cM^{^s} \to \cM^{^s}$ denote the projection. Let us define the bundle 
\[
\cT_{_{\cM}} := R^1q_*E \left(\gfr \right)
\]
on the stack $\cM^{^s}$ and $\Omega_{_{\cM}}$ as its dual. 

 The aim of our article is to describe cohomology groups $\text{H}^i \left(\cM^{^s}, \cT_{_{\cM}} \right)$ (see \eqref{thedefnos}), and $\text{H}^i \left(\cM^{^s}, \Omega_{_{\cM}} \right)$ (see \eqref{thedefnos1}),  and derive consequences on deformations of moduli spaces of torsors such as the classical theorems of Narasimhan and Ramanan \cite{nr} in positive characteristics under a mild restriction on $p$ (see \ref{main1}).  In \cite{balaji-vishwanath}, these computations were carried out for the smooth compactification of the moduli space of stable $\text{SL}(2)$-bundles. In \cite{Hitchin1987}, Hitchin gave a new proof of this result in the case of vector bundles of rank $2$ and degree $1$ and also computed the cohomology of the symmetric powers of the tangent sheaf.  We were informed by Teleman that in characteristic zero, these results (and even Hitchin's) are a direct consequence of his computations which arise from his works, \cite{bwb}, \cite{telemanfrenkel},  and \cite{fgt}. So the present paper is novel for its results in positive characteristics and also its approach. The essential new ingredient in our paper  is  the parahoric theory of torsors \cite{bs} and the landscape of {\em Parahoric-correspondences} \footnote{These are precise analogues of what are called Hecke correspondences in \cite{nr}.} (\ref{hecke}). {\em We also prove the results for all moduli stacks on stable parahoric torsors}. We add that Teleman's approach, depends heavily on his Borel-Weil-Bott results and Lie algebra computations. To the best of our knowledge, these do not generalize in any obvious manner for characteristics ${\tt p} > 0$. The stacks of parahoric torsors have been studied for quasi-split group schemes and as observed by Faltings, these have deep connections with the theory of local models for PEL Shimura varieties.  
 
 In the last section, we deal with certain classes of quasi-split group scheme $\cH$ and extend the deformation  results \eqref{quasisplitcase}  to the moduli stack $\cM^{^s}_X(\cH)$ of stable $\cH$-torsors on $X$.

\section{Group theoretical data}
Except in the last section, we assume that $G$ is a split group obtained as the base change of the Chevalley group scheme.
We shall fix a split maximal torus $T$ of $G$ and a Borel subgroup $B$ containing $T$. Let $\Phi$ be the root system relative to $(T,G)$ and $\cA$ denote the apartment of $T$ together with the origin $0$. Let ${\bf a}_0$ denote the unique alcove in $\cA$ whose closure  contains $0$ and is contained in the finite Weyl chamber determined by our chosen Borel subgroup. This determines a set $S$ of simple roots and $\mathbb{S}$ of simple {\it affine roots} $\alpha$. Let $\alpha_0$ denote the affine simple root lying outside of $S$.  Let $Y(T) := \Hom({\mathbb G}_m, T)$ the group of $1$-parameter subgroups of $T$. 
The set $S$ determines a system of positive roots $\Phi^{+} \subset {\Phi}$.  We now order the set $\Phi^+ = \left\{r_i | i = 1, \ldots, q \right\}$. We then have a family $\left\{u_{_r}:{\bg}_a \to G \mid r \in \Phi \right\}$ of {\em root homomorphisms}.

\section{Parahoric groups} 
 Let $K$ be a field equipped with a discrete valuation $v: K^{\times} \to \bz$ and we shall also assume that $K$ is {\em complete}. Let $A$ be the ring of integers, with residue field $k$. 

\subsection{\bf Parahoric subgroups}\label{boun} 
%A subset $M \subset G(K)$ is said to be {\em bounded} if for any regular function $f \in K[G]$, the values $v(f(m))$ is bounded below, when $m$ runs over all elements of $M$. In particular, we may talk of {\em bounded subgroups}. A subgroup $M \subset G(K)$ is therefore bounded if the ``order of poles" of elements of $M$ is bounded. This can be made precise by taking a faithful representation of $G \hra GL(n)$ so that elements of $M$ are represented by matrices with entries in $K$. 
For $r \in \mathbb{R}$ let $[r]$ denote the greatest integer not greater than $r$. For $\theta \in Y(T) \otimes_{\bz} {\br}$, set 
\begin{equation}
m_{r}(\theta) := -[r(\theta)]. 
\end{equation}
Denote by ${\mathfrak P}_{_\theta}(K) \subset G(K)$ the subgroup generated by $T(A)$ and the root groups $U_{r} \left(z^{m_{r}} A \right)$ for all the roots $r \in \Phi$, i.e. we have:
\begin{equation}\label{gille-1}
{\mathfrak P}_{_\theta}(K) = \left\langle T(A),~~~~ U_{r} \left(z^{m_r(\theta)}A \right), ~~~r \in \Phi \right\rangle \subset G(K).
\end{equation}
In particular, ${\mathfrak P}_{_0}(K)$ is the maximal bounded subgroup $G(A) \subset G(K)$. 

%Note that if $\theta \in Y(T)$ itself, then there exists $t \in T(K)$ such that 
%\begin{equation}\label{gille0}
%{\mathfrak P}_{_\theta}(K) = t .{\mathfrak P}_{_0}(K). t^{-1}
%\end{equation}

\subsection{\bf  Standard parahorics} \label{stdparahorics} The {\em standard parahoric subgroups} of $G(K)$ are parahoric subgroups of the canonical hyperspecial parahoric subgroup  $G(A)$. These are the inverse images under the evaluation map $$\text{ev}: G(A) \to G(k)$$ of standard parabolic subgroups $P_I \subset G$, where $I \subset S$ is any subset of the simple roots. In particular, the {\em Iwahori subgroup} ${\mathfrak I}$ is a standard parahoric and indeed, ${\mathfrak I} = \text{ev}^{-1}(B)$, $B \subset G$ being the standard Borel subgroup containing the fixed maximal torus $T$.

\subsection{Maximal parahorics} Denote by $\left\{{\alpha}^* \mid \alpha \in S \right\}$ to be the basis dual to $\{\alpha \in S\}$, i.e. $\left(\alpha^*, r \right) = \delta_{\alpha,r}$. These generate the coweight lattice $P^{\vee}$.  For every $\alpha \in S$, we define 
\begin{equation}\label{thetaalpha}
\theta_{\alpha} = \left\{\alpha^* \over c_{\alpha} \right\} \in Y(T) \otimes_{\bz} {\br},
\end{equation}
Together with $0$, these are the coordinates of $\mathbf{a}_0$. 
%The Borel Seibenthal list gives finite order elements
%\begin{equation} \label{galpha}  g_{_{\alpha}} := \alpha^* \left(\frac{1}{c_{\alpha}} \right).
%\end{equation} 
%These belong to $G(k)$ if $k$ contains roots of unity corresponding to the index of $\theta_{_{\alpha}}$ in $Y(T)$. This happens for example when the characteristic $\tt p$ of $k$ is coprime to the centre of $G$ and the coefficients $c_{_{\alpha}}$ of the highest root.
Further $\left\{{\mathfrak P}_{_{\theta_{\alpha}}}(K) \mid \alpha \in S \right\}$ and ${\mathfrak P}_{_0}(K)$ represent $G(K)$-conjugacy classes of all {\em maximal parahoric subgroups} of $G(K)$.

\subsection{\bf  Hyperspecial Parahorics} \label{hyperspecial}  Let $\alpha_{_{max}}$ denote the highest root. For $\alpha \in S$, let coefficients $c_{\alpha} \in {\bz}^+$ be defined by 
\begin{equation}\label{calpha}
\alpha_{_{max}} = \sum_{\alpha \in S} c_{\alpha} \cdot \alpha.
\end{equation} 

%with $c_{\alpha} $ and let \begin{equation} \alpha_{_0} := 1 - \alpha_{_{max}}. \end{equation}

The parahoric subgroup ${\mathfrak P}_{_{\theta_{\alpha}}}(K)$ is hyperspecial if and only if $c_{\alpha} = 1$ in the description of the long root $\alpha_{_{max}}$. Upto conjugation by $G(K)$, the hyperspecial parahorics are the following: in type $A_n$, all the $n+1$ maximal parahoric subgroups are {\em hyperspecial parahorics}; in types $B_n$ and $C_n$,  exactly $2$  maximal parahoric subgroups are hyperspecial; ${D_n}$ has exactly $4$ hyperspecial maximal parahoric subgroups; ${E}_6$ has exactly $3$; ${E}_7$ has exactly $2$ and finally the types $G_2, F_4$ and ${E}_8$ have only one hyperspecial maximal parahoric subgroup each.

\subsection{Bruhat-Tits group schemes} \label{lgpthedata}
%For $x \in X$, let $\cO_x$ denote the complete local ring at $x$ (i.e. $\cO_x \simeq k[\![t]\!]$)  and ${K}_x \simeq k(\!(t)\!)$ be its field of fractions, the corresponding local field.   
We will work with a pointed projective curve $(X,x)$. 
Let $\cG_{\alpha}$ (resp. $\mathcal{G}_{_{\theta_{_\alpha}}})$ denote the standard parahoric group scheme on $\spec A$ associated to the simple root $\alpha$ (resp. the maximal parahoric group scheme associated to vertex $\theta_{_{\alpha}}$ of the Weyl alcove). Let us denote by the same notations the group scheme on the pointed projective curve $(X,x)$ obtained by gluing $G \times (X - \{x \})$ with $\cG_{\alpha}$ (resp. $\mathcal{G}_{_{\theta_{_\alpha}}}$) via the gluing function identity. Following \cite{bs}, let us denote the moduli stack of $\cG_\alpha$-torsors by $\cM_{\alpha}$ and the stack of $\mathcal{G}_{_{\theta_{_\alpha}}}$-torsors by $\cM_{_{\theta_{_\alpha}}}$. 
 
%\beg The hyperspecial simple roots correspond precisely to the dual notion of {\em minuscule coweights}. \erem

\section{Some group-theoretical computations in the split case}
We assume $G$ to be split. We begin by recalling the following invariant.
\begin{defi} \cite[Proposition 7.2.1, page 35]{bs} Let $\mu(\alpha)$ be the set of positive roots which have as a term the simple root $\alpha$ with the highest coefficient $c_\alpha$ i.e.
\begin{equation}\label{mu}
\mu(\alpha) : = \left\{r \in \Phi^+ \mid r = c_{\alpha}.\alpha +  \sum_{\beta \neq \alpha}  x_{\beta}.\beta \right\}.
\end{equation}
\end{defi}

The quotients $${{\mathfrak P}_{_{\theta_{\alpha}}}(K) \over {\mathfrak P}^{st}_{_{\alpha}}(K)},~~~ {{\mathfrak P}^{st}_{_{\alpha}}(K) \over {\mathfrak I}}$$ are supported on the residue field $k$ and are finite dimensional. Let us describe the unipotent subgroups appearing in these quotients. 

For $\alpha \in S$ let $P_{_\alpha} \subset G$ be the maximal parabolic subgroup. Let ${\mathfrak P}^{st}_{_{\alpha}}(K) := \text{ev}^{-1} \left(P_{_\alpha} \right)$ denote the {\em standard parahoric subgroup}. Then we have the obvious inclusions:
\begin{equation}\label{inclusions}
{\mathfrak I} \subset {\mathfrak P}^{st}_{_{\alpha}}(K) = {\mathfrak P}_{_{\theta_{\alpha}}}(K) \cap {\mathfrak P}_{_{0}}(K).
\end{equation}
In this split case, let us recall the definitions of 
\begin{eqnarray*}
 {\mathfrak I} &=& \left\langle T(A), U_r(A), U_{-r}(z.A) \mid r \in \Phi^+ \right\rangle, \\
{\mathfrak P}_{_{\theta_{\alpha}}}(K) &=& \left\langle T(A), U_r \left(z^{- \left[ \left(\theta_{_\alpha}, r \right) \right]}.A \right) \mid r \in \Phi \right\rangle, \\
 {\mathfrak P}^{st}_{_{\alpha}}(K) &=& \left\langle T(A), U_r \left(z^{m_{_{r,\alpha}}}.A \right) \mid r \in \Phi \right\rangle,
 \end{eqnarray*} 
where setting $\Phi_{_{S - \alpha}} := \Big\{r \in \Phi \mid r~does~not~involve~\pm\alpha\Big\}$ we have
 
\begin{eqnarray} \label{mralpha}
m_{_{r,\alpha}} = 
\begin{cases}
1 &\text{if $r \in \Phi^{-} - \Phi_{_{S - \alpha}}$}, \\
0 &\text{if $r \in \Phi^{+} \cup \Phi_{_{S - \alpha}}$}.
\end{cases}
\end{eqnarray}

Thus, with $\mu(\alpha)$ as in \eqref{mu}, we have
\begin{equation}\label{piside}
\text{\cursive {roots}}\left({{\mathfrak P}_{_{\theta_{\alpha}}}(K) \over {\mathfrak P}^{st}_{_{\alpha}}(K)}\right) = \mu(\alpha).
\end{equation}
 Denoting $f_{\theta_\alpha}$ the concave function corresponding to $\theta_\alpha$ we observe that
\begin{eqnarray} \label{fthetaalpha}
f_{\theta_\alpha}(r):=-\left[ \left(\theta_{_\alpha}, r \right) \right] = \begin{cases} -1  \iff  \left\{r \in \Phi^{+} \mid r = c_{\alpha}.\alpha +  \sum_{\beta \neq \alpha}  x_{\beta}.\beta \right\}, \\
 0  \iff  r \in \Phi_{_{S - \alpha}}, \\
 1  \iff  r \in \Phi^{-} - \Phi_{_{S - \alpha}}. 
 \end{cases}
\end{eqnarray}

Hence, we have \small 
\[
{U_r \left(z^{- \left[ \left(\theta_{_\alpha}, r \right) \right]}.A \right) \over U_r \left(z^{m_{_{r,\alpha}}}.A \right)} \simeq {\mathbb G}_{a,k}, ~~\iff~~ \left\{r \in \Phi^{+} \mid r = c_{\alpha}.\alpha +  \sum_{\beta \neq \alpha}  x_{\beta}.\beta \right\}.
\]\normalsize
On the other hand, 
\begin{equation} \label{hside}
\text{\cursive {roots}} \left({{\mathfrak P}_{_{0}}(K) \over {\mathfrak P}^{st}_{_{\alpha}}(K)} \right) = \left\{\Phi^{-} - \Phi_{_{S - \alpha}}\right\}.
\end{equation}
In particular, when the vertex $\alpha$ is {\em hyperspecial}, $c_{\alpha} = 1$ and we get the identification 
\begin{equation}\label{piversush}
\text{\cursive {roots}} \left({{\mathfrak P}_{_{\theta_{\alpha}}}(K) \over {\mathfrak P}^{st}_{_{\alpha}}(K)} \right) = - \text{\cursive {roots}}\left({{\mathfrak P}_{_{0}}(K) \over {\mathfrak P}^{st}_{_{\alpha}}(K)} \right).\end{equation}
%\begin{flushright} {\tt Q.E.D} \end{flushright}

\subsection{For exceptional groups with trivial centre} 
Let $\alpha$ be the root to which $\alpha_{0}$ attaches itself in the extended Dynkin diagram.   We have 
\begin{eqnarray*} 2=\alpha_{_{max}}^{\vee} \left(\alpha_{_{max}} \right)= \sum_{\gamma \in S} c_{\gamma} \alpha_{_{max}}^{\vee}(\gamma) = c_{\alpha}  \alpha_{_{max}}^{\vee}(\alpha) + \sum_{\gamma \in S - \alpha} c_{\gamma} \alpha_{_{max}}^{\vee}(\gamma).
\end{eqnarray*} 
 Further, by checking the Bourbaki tables we see that $\alpha^{\vee}_{_{max}}(\alpha)=1$ and that $c_{\alpha}$ is always $2$. Thus we get $$0= \sum_{\gamma \in S - \alpha} c_{\gamma} \alpha^{\vee}_{_{max}}(\gamma).$$ Since $\gamma \in S$ and $\alpha_{_{max}}$ are positive roots, $\alpha_{_{max}}^{\vee}(\gamma) $ is non-negative. Thus
 $$0 =\alpha^{\vee}_{_{max}}(\gamma) \quad \text{for} \quad \gamma \in S - \alpha.$$
For $\beta \in \Phi^{-}$, let $c_{\alpha}^{\beta} $ denote the coefficient of $\alpha$ when we write $\beta$ in terms of simple roots in $S$. Thus $c_{\alpha}^{\beta} \in \{0, -1,-2 \}$ and 
\begin{equation} \label{calphabeta}
\alpha^{\vee}_{_{max}}(\beta)= c^{\beta}_{\alpha}.
\end{equation}

\subsection{Parahoric-correspondence and some geometry} \label{hecke} 
 The inclusions \eqref{inclusions} of groups give morphisms 
 $$\cG_\alpha \to G_A, \quad \cG_\alpha \to \mathcal{G}_{_{\theta_{_\alpha}}}$$
 of  $A$-group schemes. These morphisms naturally extend to morphisms over the pointed curve $(X,x)$. This immediately gives the following diagram of Parahoric-correspondence on the moduli stacks (see \cite[8.2.1]{bs}, \cite[\S 7.6, (7.6.3)]{me}):
\begin{equation}\label{heckecorresp} \begin{tikzcd}
	& {\cM_\alpha} \\
	{\cM_{_{\theta_{_\alpha}}}} && {\cM} 
	\arrow["\pi", from=1-2, to=2-1]
	\arrow["h"', from=1-2, to=2-3]
\end{tikzcd}\end{equation}
Recall $(X,x)$ is our pointed projective curve. 
Let  $\mathcal{G}_{_{\theta_{_\alpha},x}}$ denote the $x$-fibre  of the group scheme $\mathcal{G}_{_{\theta_{_\alpha}}}$ and let $\mathcal{G}^{^{\tt{red}}}_{_{\theta_{_\alpha},x}}$ denote its reductive quotient. Let  $P_{_{\theta_{_\alpha}},x}$ denote the maximal parabolic subgroup corresponding to the deletion of $\alpha_0$ in $\mathbb{S} - \alpha$  where we recall that $\mathbb{S}$ denotes the set of affine simple roots. 
The basic observation is that $\pi$ is an \'etale locally trivial representable fibration with fibres 
\begin{equation} \label{basicobservation} \mathcal{G}^{^{\tt{red}}}_{_{\theta_{_\alpha},x}}/P_{_{\theta_{_\alpha}},x},
\end{equation}  (see \cite[4.3.12]{pp} for a more general result). Similarly, the $h$-fibre is isomorphic to $$G/P_\alpha,$$ where $P_\alpha \subset G$ is the maximal parabolic subgroup of $G$ containing $B$. 

Let us elaborate on the above fibrations. Let $E \ra X \times \cM$ denote the universal bundle and $E_{\theta_\alpha} \ra X \times \cM_{\theta_\alpha}$ denote the universal torsor. Let $E_x$ and $E_{\theta_\alpha,x}$ denote their restrictions to $x \times \cM$ and $x \times \cM_{\theta_\alpha}$ respectively. Let us view $\mathcal{G}^{^{\tt{red}}}_{_{\theta_{_\alpha},x}}/P_{_{\theta_{_\alpha},x}}$ as homogenous space for the group $\mathcal{G}_{_{\theta_{_\alpha},x}}$. Then we have 
\begin{eqnarray} \label{ModStackasE(G/P)}
h: \cM_{\alpha} \ra \cM & = & E_x(G/P_{\alpha}) \ra \cM, \\ \pi: \cM_{\alpha} \ra \cM_{\theta_\alpha} &= & E_{\theta_\alpha,x} \times^{\mathcal{G}_{_{\theta_{_\alpha},x}}} \left( \mathcal{G}^{^{\tt{red}}}_{_{\theta_{_\alpha},x}}/P_{_{\theta_{_\alpha},x}} \right) \ra \cM_{\theta_\alpha}. \end{eqnarray}

The group scheme $\cG_\alpha$ is obtained from both the hyperspecial group scheme $G_A$ as well as  the maximal parahoric group scheme $\mathcal{G}_{_{\theta_{_\alpha}}}$ by a {\em N\'eron dilatation} (see \cite[\S 2.2]{bpconformal}). In fact, the group scheme $\cG_\alpha$ is the N\'eron blow-up of $\mathcal{G}_{_{\theta_{_\alpha}}}$ 
along the inverse image of the subgroup $P_{_{\theta_{_\alpha}},x}$ under the quotient map $\mathcal{G}_{_{\theta_{_\alpha},x}} \rightarrow                                             \mathcal{G}^{^{\tt{red}}}_{_{\theta_{_\alpha},x}}$.

%the inverse image $P'_{_{\theta_{_\alpha}},x} \subset \mathcal{G}_{_{\theta_{_\alpha}},x}$ of

\subsection{Relations between universal torsors} Let $E \to X \times \cM$ (resp. $\cE_\alpha \to X \times \cM_\alpha$) be the universal $G_X$-torsor (resp. $\cG_\alpha$-torsor). The inclusions \eqref{inclusions} give canonical morphisms $\cG_\alpha \to G_X$. Then we have the following relations:
\begin{enumerate}\label{pullbacks}
\item[a)] $\cE_{\alpha} \times^{\cG_\alpha} G \simeq (1 \times h)^*(E)$,
\item[b)] $\cE_{\alpha}(\gfr) \simeq (1 \times h)^* \left(E(\gfr) \right)$.
\end{enumerate}

\section{The Kodaira-Spencer map}
We recall a lemma from \cite[Proposition 4.4]{nr}. Fix a point $x \in X$.
\begin{lem}\label{propo4.4} Let $\{V_s\}_{_{s \in S}}$ be a family of vector bundles on $X$ parametrized by $S$ and let $\tcq:\mathbb P(V^*_x) \to S$ be the projection. Let  $\tau_x$ denote the torsion sheaf on $X$ of height $1$ at $x \in X$. Consider the vector bundle $\cK^*$ on  $X \times \mathbb P \left(V^*_x \right)$ defined by the following exact sequence:
\begin{equation}\label{tauto}
0 \to \cK^* \to \left(1 \times \tcq \right)^* \left(V^* \right) \to p_{_{\mathbb P \left(V^*_x \right)}}^* \cO_{\tcq}(1) \otimes p_{_X}^* \left(\tau_x \right) \to 0.
\end{equation}
 Then the Kodaira-Spencer map for the family $\left\{\cK_y \right\}_{_{y \in X}}$ of vector bundles parametrised by $X$:
\begin{equation}
\rho_y:\cT_{_{y,X}} \to \text{H}^1 \left( \mathbb P \left(V^*_x \right),\text{ad}_y(\cK) \right)
\end{equation}
 is injective at the point $y = x$.
\end{lem}

\begin{proof} We outline an argument for the case when $\text{rank}(V) = 2$. The proof lies in showing that the map $\rho_x$ is non-zero, and injectivity follows.  
Let $R := {\hat\cO}_x \simeq k[\![ t ]\!]$. Let $B := R \left[t_0,t_1 \right]$. Then \eqref{tauto} reduces to an exact sequence of graded $B$-modules:
\begin{equation}
0 \to M \to B \oplus B \stackrel{f}\to k[t_0,t_1] \to 0
\end{equation}
where for $h_j \in B$, denoting their images in $k\left[t_0,t_1 \right]$ by $\overline{h_j}$, we have $f \left(h_0,h_1 \right) = \overline{h_0}.t_0 + \overline{h_1}.t_1$. It can be checked that $M$ is generated by $$E_0 = \left(-t_1,t_0 \right), E_1 = (t,0) \quad \text{and} \quad E_2 = (0,t)$$ and these satisfy the unique relation $$t.E_0 + t_1.E_1 + (-1)t_0.E_2 = 0.$$ We then get the exact sequence:
\begin{equation}
0 \to B(-1) \stackrel{i}\to M \stackrel{j} \to B \to 0
\end{equation}
where $j(E_0) = 0, j(E_1) = t_0, j(E_2) = t_1$. So the family $\cK^*$ in an analytic neighbourhood of $x$ gets identified with a family of extensions of $\cO(1)$ by $\cO(-1)$ where $x$ corresponds to the split extension. The infinitesimal deformation map for this family of extensions is then checked to be non-zero.    

\end{proof}

\begin{thm}\label{nrdef}
Let $E$ be the universal family on  $X \times \cM$, then the Kodaira-Spencer map, $$\rho_X: \cT_{X} \ra R^1 p_{_{X,*}} E(\gfr),$$
is injective. Further, $$\rho_x: \cT_{x,X} \ra H^1 \left(\cM_{\alpha}, \cE_{\alpha}(\gfr) \right)$$ is also injective.
\end{thm}
\begin{proof}
Let $W \ra X \times S$ be a family of vector bundles of degree zero. Let $x \in X$ be a closed point. Set $T:=\mathbb{P} \left(W_x \right)$ and $\phi: T \ra S$ be the natural projection. Let $$W_{_\phi}:= (1 \times \phi)^*(W).$$ This defines a family of bundles on $X$ parametrized by $T$. Let $\tau_x$ be the torsion sheaf of height $1$ on $X$ supported at $x$ and $p_{_T}: X \times T \to T$ the projection. 

We then have the universal quotient morphism $$W_{_\phi} \ra p_{_T}^* (\cO_{_\phi}(1)) \otimes p_{_X}^* \left(\tau_x \right)$$ on $X \times T$, whose kernel we denote as $K$. Whence, we have the following short exact sequence of sheaves on $X \times T$:
\begin{equation}
0 \ra K \ra W_{_\phi} \ra p_{_T}^* (\cO_{_\phi}(1)) \otimes p_{_X}^* \left(\tau_x \right) \ra 0
\end{equation} 
where $K$ is a family of vector bundles of degree $-1$. The bundle $K$ is classically called {\it an elementary transformation of $W_{_\phi}$ at $x$}.

We now consider $\psi: \mathbb{P} \left(K_x^* \right) \ra T$. 
Then the elementary transformation of $(1 \times \psi)^* \left(K^* \right)$ at $x$ is $(1 \times \psi)^* \left(W_{_\phi}^* \right)$. We denote its dual as  $$W_{_{\psi\phi}}.$$ By \eqref{propo4.4}, we have the injection: 
\begin{equation}\label{inj1}
\rho_x: \cT_{x,X} \ra H^1 \left(\mathbb{P} \left(K_x^* \right), ad_x\left(W_{_{\psi\phi}} \right) \right) .
\end{equation} 
We also have isomorphisms
\begin{equation}\label{inj2}
H^1 \left(\mathbb{P} \left(K_x^* \right), ad_x\left(W_{_{\psi\phi}} \right) \right) \simeq H^1 \left(T, ad_x W_{_{\phi}} \right) \simeq H^1 \left(S, ad_x W \right).
\end{equation}
%\begin{equation} \end{equation} 
Let $p_{_X}: X \times S \ra X$ denote the first projection. Whence, by using the injection \eqref{inj1} and the identifications \eqref{inj2}, we have an injection $\rho_x: \cT_x \ra H^1(S, ad_x W)$ at each $x$ and hence an injection
\begin{equation}
\rho: \cT_X \ra R^1 p_{_{X,*}} (\text{ad}~W).
\end{equation}  
Let $E \ra X \times S$ be a principal $G$-bundle. We fix an faithful representation $G \hookrightarrow SL(W)$. By the associated construction this gives us the family of vector bundles $E(W) \ra X \times S$ such that $E(W)_s$ has degree $0$ for each $s \in S$.

We have a natural commutative diagram by the canonicity of the Kodaira-Spencer map:
\begin{equation}
\xymatrix{
\cT_X \ar[rr]^{\rho_G} \ar[rrd]_{\rho_W} && R^1 p_{_{X,*}} E(\gfr) \ar[d] \\
&& R^1 p_{_{X,*}} \left(\text{ad}~E(W) \right)
}
\end{equation}
Since $\rho_W$ is injective it follows that $\rho_G$ is an injective. 

Since $\rho_x: \cT_{x,X} \ra  H^1 \left(\cM_{\alpha}, \cE_{\alpha}(\gfr)  \right)$ factors via $H^1 \left(\cM, E(\gfr) 
\right)$ by \eqref{pullbacks}, the second injectivity now follows.
\end{proof}

%\begin{equation} \cM_{\alpha}:= \cM \left(\cG_{\alpha} \right).
%\end{equation}
\section{Closer look at Parahoric-correspondences}
\subsection{Some general considerations on flag bundles} Let $E \to T$ be a principal $G$-bundle and let $P \subset G$ be a parabolic subgroup given by the subset $I \subset S$. Let $\mathfrak g$ and $\mathfrak p$ be the respective Lie algebras and $E(\gfr/\pfr)$ be the associated fibration.  
\begin{thm} \label{tangentspacethm} Let $\tcf:E(G/P) \to T$ be the fibration with fibres $G/P$. If $\cT_{_{\tcf}}$ denotes the relative tangent sheaf, then we have the following canonical isomorphisms: 
\begin{equation}\label{tgt}
R^j{\tcf}_* \left(\cT_{_{\tcf}} \right) = 
\begin{cases}
E(\gfr) &j = 0 \\
0 &\text{otherwise}
\end{cases}
\end{equation}
except in three {\em exceptional} cases in the sense of Demazure \cite{demazure1977} (see also \cite[Theorem 2, page 131]{Akhiezer1995LieGA}). And
\begin{equation}\label{cotgt}
R^j{\tcf}_* \left(\cT_{_{\tcf}}^* \right) =
\begin{cases}
\cO_X^{\ell} & j = 1 \\
0 & j \neq 1
\end{cases}
\end{equation}
where $\ell$ is the number of Weyl group elements $w$ of length one such that for every $\beta \in I$, $w \beta$ is also a positive root. In particular, when $P$ is a maximal parabolic subgroup associated to  the simple root $\alpha$, i.e., $I = S - \alpha$, then $\ell = 1$, and the set defining it is the simple reflection ${\text\cursive s}_\alpha$.
\end{thm} 
\begin{proof} The proof of \eqref{tgt} follows immediately from a classic result due to J. Tits \cite{tits62}. For positive characteristics see \cite[Proposition 2, page 182]{demazure1977}. For an exposition, see Akhiezer \cite{Akhiezer1995LieGA}. The proof of \eqref{cotgt} follows from a theorem due to Marlin \cite{Marlin1977} over fields of characteristic zero. For positive characteritics, see Jantzen \cite[Page 245]{Jantzen}.
\end{proof}
\begin{rem}\label{exceptionalcases} Two of the three {\em exceptional} cases are certain non-maximal parabolic subgroups in types $\tt B_{_n}$ (with $n \geq 3$) and $\tt C_{_n}$ (with $n \geq 2$).  The remaining exceptional case is in the group $\tt G_{_2}$ coming from the maximal parabolic subgroup associated to the simple root $\alpha_2$ away from the extended vertex in the extended Dynkin diagram. 

Furthermore, if all root lengths are the same then these cases are all {\em non-exceptional} cases.  
\end{rem}

\subsection{The cohomology computations}
\begin{thm} \label{key1} For $G \neq G_2,F_4,E_8$, let $\theta_\alpha$ be a hyperspecial vertex different from the origin $0$. Let $\cT_\pi$ (resp. $\cT_h$) be the relative tangent sheaves. As locally free sheaves on $\cM_\alpha$, we have the following canonical isomorphism:
\begin{equation}\label{nara1}
\cT_{\pi}^* \simeq \cT_h.
\end{equation}
Whence,  $H^i \left(\cM_{\alpha}, \cT_h \right) \simeq H^{i-1} \left(\cM_{\theta_{_{\alpha}}}, \cO \right)$.
\end{thm}
\begin{proof} Let $\cE_\alpha$ denote the universal $\cG_\alpha$-torsor on $X \times \cM_\alpha$. Since $\cG_\alpha$ is a standard parahoric group scheme, it comes with a canonical inclusion $\cG_\alpha \to G_X = G \times_k X$ on $X$. So we can take the associated $G_X$-torsor $$E_\alpha := \cE_\alpha \times^{\cG_\alpha} G$$ on $X \times \cM_\alpha$.

Let $\mathfrak g := \text{Lie}(G)$ and $\mathfrak p_\alpha:= \text{Lie}(P_\alpha)$.
We consider its restriction $E_{\alpha,x}$ which is a $G$-bundle on $x \times  \cM_\alpha$. Let us take the associated fibration $$E_{\alpha,x}\left( \mathfrak g/\mathfrak p_\alpha \right):= E_{\alpha,x} \times^{G} \mathfrak g/\mathfrak p_\alpha$$ for the adjoint action.
It is not hard to check that we have an isomorphism:
\begin{equation}
\cT_h \simeq E_{\alpha,x}\left( \mathfrak g/\mathfrak p_\alpha \right).
\end{equation}
We have a similar homomorphism $\cG_\alpha \to \cG_{_{\theta_{_\alpha}}}$ induced by the inclusion of the parahoric subgroup \eqref{inclusions}. We can take the extension of structure group of $\cE_\alpha$ by this homomorphism and get the $\cG_{_{\theta_{_\alpha}}}$-torsor $\cE_{_{\theta_{_\alpha}}}$ on $X \times \cM_\alpha$ and its restriction $\cE_{_{\theta_{_\alpha},x}}$ to $x \times \cM_\alpha$.

With notations as in \S \ref{hecke}, let $\mathfrak g_{_{\theta_{_\alpha}}} := \text{Lie}\left(\mathcal{G}^{^{\tt{red}}}_{_{\theta_{_\alpha},x}} \right)$ and let $\mathfrak p_{_{\theta_{_\alpha}}} := \text{Lie}\left(P_{_{\theta_{_\alpha}}} \right)$. Then by \eqref{basicobservation}, as for $\cT_h$ we have the identification
\begin{equation}
\cT_\pi \simeq \cE_{_{\theta_{_\alpha},x}}( \mathfrak g_{_{\theta_{_\alpha}}}/\mathfrak p_{_{\theta_{_\alpha}}}).
\end{equation}
It follows that \eqref{hside} (resp. \eqref{piside}) give the set of roots defining $E_{\alpha,x}\left( \mathfrak g/\mathfrak p_\alpha \right)$ (resp. $\cE_{_{\theta_{_\alpha},x}}( \mathfrak g_{_{\theta_{_\alpha}}}/\mathfrak p_{_{\theta_{_\alpha}}})$). 
The identification \eqref{piversush} gives the isomorphism \eqref{nara1}.

By \eqref{cotgt} we have 
\begin{equation}\label{oneshift}
H^i \left(\cM_{\alpha}, \cT_h \right) \simeq H^i \left(\cM_{\alpha}, \cT^*_\pi \right) \simeq H^{i-1}\left(\cM_{\theta_{_{\alpha}}}, \cO \right).
\end{equation}

\end{proof}

\begin{thm} \label{key2} Let $G= G_2,F_4$ or $E_8$, when $0$ is the only hyperspecial vertex. Let $\alpha \in S$ be the unique simple root to which $\alpha_0$ connects in the extended Dynkin diagram. The coefficient of $\alpha$ in $\alpha_{_{max}}$ is always $2$ in these three cases, and we have $\mu(\alpha)$ (see \eqref{mu}) is the singleton set $\left\{ \alpha_{_{max}} \right\}$. Moreover, the fiber of $\pi$ in the Parahoric-correspondence diagram \eqref{heckecorresp} is canonically isomorphic to $$SL \left(\alpha_{_{max}} \right)/B \left(\alpha_{_{max}} \right) \simeq \mathbb{P}^1.$$ 
 We have $H^i \left(\cM_{\alpha}, \cT_h \right) \simeq H^{i-1} \left(\cM_{\theta_{_{\alpha}}}, \cO \right)$.
\end{thm} 
\begin{proof} The first two assertions follow by an examination of the Bourbaki Tables. The third follows by \eqref{piside} after noting that $\mu(\alpha)$ is the singleton $\left\{ \alpha_{_{max}} \right\}$. By \eqref{hside}, the sheaf $\cT_h$, restricted to the fibers of $\pi$,  decomposes into line bundles of degrees given by $\alpha_{_{max}}^{\vee}(\beta)$ for $\beta \in \Phi^{-} - \Phi_{S - \alpha}$. By \eqref{calphabeta}, all these degrees are $-1$ except for $- \alpha_{_{max}}$ for which it is $-2$. Thus $R^0 \pi_* \cT_h=0$ and $R^1 \pi_* \cT_h=\cO$. The last assertion follows by an application of the  Leray spectral sequence.
\end{proof} 

\begin{Cor}\label{rationality}   By \eqref{key1}, \eqref{key2}, for all $G$, we have
\begin{equation}\label{irrationality}
H^i \left(\cM_{\alpha}, \cT_h \right) \simeq H^{i-1}\left(\cM_{\theta_{_{\alpha}}}, \cO \right) .
\end{equation}
\end{Cor}
\begin{rem} In particular, the exceptional case mentioned in \eqref{tangentspacethm} and discussed in \eqref{exceptionalcases} do not occur for types $F_4$ or $E_8$, while for $G_2$, the exceptional case $\alpha_2$ is the one which is {\em not connected} to $\left\{ \alpha_{_{max}} \right\}$.
\end{rem}
%Therefore 
%\begin{equation} \label{key2application} 
%H^i \left(\cM_{\alpha}, \cE_{\alpha,x}(\gfr) \right) = \begin{cases} 0 \quad \text{for} \quad  i=0,2,3 \\ 
                                                              %1 \quad \text{for} \quad  i=1.
                                                              %\end{cases}
%\end{equation}
%This holds for all closed points $y \neq x$ also.

\begin{thm} \label{2.4} Let $p: X \times \cM \ra X$ denote the first projection. Then we have
\begin{equation*}\label{thedefeqn}
R^i p_* \left(E(\gfr) \right) =
\begin{cases}
0 & \quad i \neq 1 \\
\cT_{_{X}} & \quad i=1.
\end{cases}
\end{equation*} 
In particular,  the Kodaira-Spencer map $\rho_X$ of \eqref{nrdef} is an isomorphism in degree $1$.
\end{thm}
\begin{proof} 
%Modulo the second part of \eqref{nrdef}, it is enough to show
%\begin{equation*}
%\text{H}^i \left( \cM_\alpha, \cE_{\alpha,x}(\gfr) \right)  \simeq \text{H}^{i-1} \left( \cM_{_{\theta_{_\alpha}}}, \cO \right) \otimes \cT_{_{X,x}}
%\end{equation*}
%Let $\cM_{\alpha}^{s}=h^{-1}(\cM(G)^{s})$ be the inverse image of the open substack of {\em regularly stable} $G$-torsors. It is well-known that the codimension of the complement of the regularly stable locus is $g - 1$. We may therefore work with the open subset $\cM_{\alpha}^{s} \subset \cM_{\alpha}$ and use Hartogs' lemma when needed, with the precise cohomology bounds. The assumption $g \geq 5$ has been made precisely for this purpose. By Hartogs, for $i \leq 2$, we have $H^i \left(\cM_{\alpha}^{s}, \cE_{\alpha,x}(\gfr) \right) = H^i \left(\cM_{\alpha}, \cE_{\alpha,x}(\gfr) \right)$. 
  
%By \eqref{pullbacks}, we have $$\cE_{\alpha,x}(\gfr) \simeq h^*(E_{x}(\gfr)).$$
We begin by computing  the dimension of the fibers of  $R^i p_* \left(E(\gfr) \right)$. By \eqref{ModStackasE(G/P)}, we may apply Theorem \ref{tangentspacethm} to the fibration $h: \cM_{\alpha} \ra \cM$. By the Leray spectral sequence, applying Theorem \ref{tangentspacethm} first for $j=0$ and then for $j>0$ for any $x \in X$ we have
\begin{eqnarray}
\text{H}^i \left(\cM, E_x(\gfr) \right) \stackrel{\eqref{tgt}} = \text{H}^i \left(\cM, h_*(\cT_h) \right)  \stackrel{\eqref{tgt}} = \text{H}^i \left( \cM_{\alpha}, \cT_h \right)  \stackrel{\eqref {irrationality}} \simeq \text{H}^{i-1} \left( \cM_{_{\theta_{_\alpha}}}, \cO \right).
\end{eqnarray} 
Note that $\text{H}^{i-1} \left( \cM_{_{\theta_{_\alpha}}}, \cO \right) \simeq \text{H}^{i-1} \left( \cM_{\alpha}, \cO \right) \simeq \text{H}^{i-1} \left( \cM, \cO \right)$ since the morphisms in the Parahoric-correspondence are $G/P$-type fibrations.
 
We claim that Borel-Weil-Bott for the stack $\cM$ and the line bundle $\cO$ holds. Let 
$L_{X^\circ}(G)$ denote the ind-scheme parametrising regular maps from $X^{\circ}:=X - \{x \}$ to $G$.
Let $$ \mathcal{F} \ell:=LG/L^+{G} \rightarrow \mathcal{M}= \mathcal{F} \ell/L_{_{X^\circ}}(G) $$ denote the affine flag variety serving as an atlas. Then we have a spectral sequence $$E^{^{pq}}=H^{^p} \left(L_{_{X^\circ}}(G), H^q \left( \mathcal{F} \ell, \mathcal{O} \right) \right) \implies H^{^n}(\mathcal{M},\mathcal{O}).$$ This immediately gives $\text{H}^{^{i}}(\mathcal{M},\mathcal{O}) =k$ for $i=0$ since the affine flag variety $\mathcal{F} \ell$ is ind-proper by \cite{pradv}. For higher $i$, one has to check that $H^{^p} \left(L_{_{X^\circ}}(G), \mathcal{O} \right)$ vanishes for higher $p$. The  argument in \cite[Prop 6.1.1 (1)]{me} in the analytic setup generalises: one writes $$L_{_{X^\circ}}(G)= \varinjlim Y_{_n}$$ as a direct limit of affine schemes $Y_{_n}$ (see \cite{bwb} and \cite[Prop 6.0.3]{me}) where each morphism $Y_{_n} \hookrightarrow Y_{_{n+1}}$ is a closed immersion. The arguments in the proof of \cite[Prop 5.0.1]{me} show that the natural map $$\Gamma(\mathbb{N}): Ab(\mathcal{F} \ell) \rightarrow Func(\mathbb{N},Ab)$$ from the category $Ab(\mathcal{F} \ell)$ of abelian sheaves on $\mathcal{F} \ell$ to the category of contravariant functors $Func(\mathbb{N},Ab)$ from the category $\mathbb{N}$ to abelian groups $Ab$ admits a functorial and exact left-adjoint. Thus $\Gamma(\mathbb{N})$ maps injectives to injectives. Therefore, for any abelian sheaf on any Grothendieck site of $L_{_{X^\circ}}(G)$ one gets the following Grothendieck spectral sequence 
$$R^{^p} \varprojlim_n H^{^q} \left(Y_{_n}, \mathcal{F} \right) \implies H^{^*} \left(L_{_{X^\circ}}(G), \mathcal{F} \right).$$
For $q \geq 1$, the groups $H^{^q} \left(Y_{_n}, \mathcal{O} \right)$ vanish. For $q=0$, the inverse system $\cdots \leftarrow H^0 \left(Y_{_n},\mathcal{O} \right) \leftarrow H^0 \left(Y_{_{n+1}},\mathcal{O} \right) \leftarrow \cdots$ is surjective on each arrow and hence satisfies Mittag-Leffler conditions proving that its higher $\varprojlim$ is zero.

Thus by BWB for the pair $(\mathcal{M},\mathcal{O})$ we get  $\text{H}^{i-1} \left( \cM, \cO \right) = 1$ for $i =1$ and $0$ elsewhere and the result follows for $i \neq 1$. Moreover, the dimension of the fibers of  $R^1 p_* \left(E(\gfr) \right)$ becoming to $1$ together with \eqref{nrdef} shows $\rho_X$  is an isomorphism since it is already injective by \eqref{nrdef}.
\end{proof}
\begin{prop} We have
\begin{equation}\label{thenumbers1}
\text{H}^i \left(X \times \cM, E(\gfr) \right) = \text{H}^{i-1} \left(X, \cT_{_X}\right).
\end{equation} 
\end{prop}
\begin{proof} This is immediate by considering the first projection $p: X \times \cM_{\alpha} \ra X$ and the degenerate  Leray spectral sequence, together with \eqref{thedefeqn}.
\end{proof}

\begin{thm}\label{main1} Let the characteristic $\tt p$ of $k$ be coprime to the order of the centre of $G$ and the coefficients of the highest root.
Let $\cM^s \subset \cM$ denote the open locus of stable torsors. {\em For $g \geq  4$, for all $G \neq \text{SL}(2)$, and when $G = \text{SL}(2)$, for $g \geq 5$}, we have a canonical isomorphism
\begin{equation}\label{thedefnos}
\kappa:\text{H}^i \left(X, \cT_{_X} \right) \simeq \text{H}^i \left(\cM^{^s}, \cT_{_\cM} \right)  \quad \text{for} \quad i=0,1, 2.
\end{equation}
In particular, for $i = 0,2$ the group $H^i \left(\cM^{^s}, \cT_{_{\cM}} \right)$  vanishes. For $i = 1$, the map $\kappa$ is the Kodaira-Spencer map associated to the deformations of the moduli stack $\cM^{^s}$. 
\end{thm} 
\begin{proof} Let us restrict to the stable locus $\cM^s$ and consider the second projection $q:X \times \cM^s \to \cM^s$. Here we have
\begin{equation}
R^i q_* \left(E(\gfr) \right) =
\begin{cases}\label{justfornow}
0 & \quad i =0~$\&$~ i \geq 2 \\
\cT_{_{\cM^s}} & \quad i=1.
\end{cases}
\end{equation} 
The cases $i \geq 2$ holds because the dimension of the fiber of $q$ is one. The case $i = 1$ follows by definition. 

In characteristic zero, the case $i=0$ follows by {\em stability} of the family $\left\{E_t \right\}_{_{t \in \cM^s}}$ since $H^0(E_t(\gfr))=0$ when $G$ is semisimple. In positive characteristics, we need to exercise caution, especially since Lie algebras do not capture the complexity. Observe that the automorphism groups of stable, non-regularly stable torsors are finite  and viewing the automorphisms as a subgroup of $G$ under evaluation at the base point $x \in X$, one sees that their centralisers in $G$ are semisimple. It is easy to check that one gets a  reduction of structure group of these stable bundles to these semisimple subgroups (see for example \cite[Prop 2.4]{bbn2005}). Moreover, groups arise from the Borel-de Seibenthal list of maximal rank subgroups of $G$, and by our assumptions on the characteristic $\tt p$, we avoid small characteristics we land in a situation where characteristic zero notions work for us.

By  Biswas-Hoffmann  \cite[Proof of Lemma 2.1]{BHf1}  (see also Faltings \cite[Theorem II.6]{faltings1993}), the  codimension of the complement of the stable locus in the stack $\cM$ is bounded below by $(g-1)\left(\text{dim}(G) - \text{dim}(P)\right)$, for $P \subset G$ running over  parabolic subgroups of $G$. Hence the required codimension is bounded below by 
\begin{equation} \label{dimensionestimation}  (g - 1)\left(\text{dim}(G) - \text{dim}(P)\right) \stackrel{\eqref{hside}} \geq (g-1) \mid \Phi^{-} - \Phi_{_{S - \alpha}}\mid
\end{equation} for  maximal parabolics $P_\alpha$.  If $G \neq \text{SL}(2)$, this is at least $2(g-1)$ and  hence, for $g \geq 3$, this is codimension is $\geq 4$. Else, when we allow $\text{SL}(2)$, we impose $g \geq 5$ and the lower bound is $\geq 4$ again.  

By using the well-known theorem (of Hartogs-type) on
extendability of cohomology classes for $E(\gfr)$ and the inclusion $X \times \cM^{^s} \subset X \times \cM$, and the Leray spectral sequence and \eqref{justfornow}, we conclude (for $i = 0,1,2$), 
\begin{equation}\label{thebigdef}
\text{H}^i(\cM^s,\cT_{_{\cM}}) = \text{H}^{i+1} \left( X \times \cM^s, E(\gfr) \right)  = \text{H}^{i+1} \left( X \times \cM, E(\gfr) \right) \stackrel{\eqref{thenumbers1}} = \text{H}^i(X,\cT_{_X}).
\end{equation}

%The result follows from the equality $\text{H}^i(X \times \cM^s, p^*\cT_{_X}) = \text{H}^i(X,\cT_{_X}).$
 
%By \eqref{thebigdef}, the obstruction to existence of deformations lies in $ \left(\text{H}^0 (X, \cT_{_X}) \otimes \text{H}^2(\cM, \cO)\right) \oplus \left(\text{H}^1 (X, \cT_{_X}) \otimes \text{H}^1(\cM, \cO)\right)$. Now $\text{H}^0 (X, \cT_{_X})=0$ because the automorphism group of the curve is finite and $\text{H}^1(\cM, \cO) = 0$ because the Picard group of $\cM$ is discrete. 

\end{proof} 
\begin{Cor}
Let $\cM^{^s}_{_{\theta_{_\alpha}}}$ be the stack of torsors for a maximal hyperspecial parahoric group scheme $\cG_{_{\theta_{_\alpha}}}$. Then with same hypothesis on $g$, we have:
\begin{equation}
\text{h}^i \left(\cM^{^s}_{_{\theta_{_\alpha}}}, \cT_{_{\cM_{_{\theta_{_\alpha}}}}} \right) = \begin{cases} 0 & i=0,2, \\
3g-3 & i=1,
\end{cases}
\end{equation}
In particular, it holds for the moduli stack of vector bundles of rank $n$ with fixed determinant. Moreover, if the degree is co-prime to $n$, then semistable bundles are in fact {\em regularly stable}, and we get the desired dimensions for the fine moduli space which is smooth and projective. This is the main theorem of \cite{nr} (see also \cite{Hitchin1987}).
\end{Cor} 
\begin{proof} The proof is gotten by reversing the process carried out for $\cM$. More precisely, we reverse the roles of the maps $h$ and $\pi$ in the situation of \ref{key1} to obtain the isomorphism $H^i \left(\cM_{\alpha}, \mathcal{T}_{_{\pi}} \right) \simeq H^{i-1} \left(\mathcal{M},\mathcal{O} \right)$. Now \ref{2.4} and \ref{thenumbers1} generalise for $\mathcal{M}_{\theta_{\alpha}}$. The generalisation of the dimension estimation \eqref{dimensionestimation} is provided in \cite[Lemma 7.7.2 and 7.7.3]{pp} by analysing the case $G \neq SL_2$ and $G=SL_2$ separately for (loc. cit.) Lemma 7.7.3. Thus \ref{main1} generalises to $\mathcal{M}_{\theta_{\alpha}}$ as well. Finally, note that vector bundles of rank $n$ and determinant of degree $-d$, where $0 \leq d < n$, correspond precisely to the moduli space of parahoric torsors with parahoric structure at a base point of type given by the simple root $\alpha=\alpha_d$ (see \cite{bs} and  \cite[\S 10.2]{me}).   \end{proof}

\section{The cotangent bundle and its cohomology}
Let $$\Omega_{_{\cM}} := \cT_{_{\cM}}^*.$$

\begin{thm}\label{main2} Let the genus and characteristics assumptions be as in \ref{main1}. Then we have the following:
\begin{equation}\label{thedefnos1}
\text{H}^j \left(\cM^{^s}, \Omega_{_{\cM}} \right) = \text{H}^{j-1} \left(X, \cO_X \right ) 
\end{equation}
for $j = 0,1,2$.
\end{thm}
\begin{proof} Since $E_t(\gfr)$ is self-dual, by Serre duality on $X$, we have:
\[
\text{H}^1\left(X, E_t(\gfr) \right)^* \simeq \text{H}^{0} \left(X, K_X \otimes E_t(\gfr) \right),
\]
\[
\text{H}^0 \left(X, E_t(\gfr) \right)^* \simeq \text{H}^{1} \left(X, K_X \otimes E_t(\gfr) \right)
\]
where $K_X$ is the canonical bundle on $X$.
Consider the two projections $p:X \times \cM \to X$ and $q:X \times \cM \to \cM$ and again restrict to the stable locus $\cM^s$. Then Serre duality now gives the following isomorphism:
\begin{equation}
R^{^j}q_* \left( p^* \left(K_X \right) \otimes E(\gfr) \right) 
 =
\begin{cases}
0 & \quad j > 0 \\
\Omega_{_{\cM}} & \quad j = 0.
\end{cases}
\end{equation}
Therefore, by the Leray spectral sequence for $q$, we get the identification:
\begin{equation}
\text{H}^j \left(\cM^s, \Omega_{_{\cM}} \right) = \text{H}^j\left(X \times \cM^s, p^*(K_X) \otimes E(\gfr)\right).
\end{equation}
By the well-known theorem (of Hartogs-type) with the genus bounds as in \ref{main1}, we have:
\begin{equation}\label{herehartogs}
\text{H}^j\left(X \times \cM^s, p^*(K_X) \otimes E(\gfr)\right) = \text{H}^j\left(X \times \cM, p^*(K_X) \otimes E(\gfr)\right)
\end{equation}
for $j = 0,1,2$.

On the other hand, by the Leray spectral sequence for $p$ together with \ref{2.4} and the projection formula, (when char $\tt p$ has bounds as in \ref{main1}), we get:
\begin{equation} 
\text{H}^j\left(X \times \cM, p^*(K_X) \otimes E(\gfr)\right) = \text{H}^{j-1}(X, K_X \otimes \cT_X) = \text{H}^{j-1}(X, \cO_X).
\end{equation}
for all $j$. 
Thus, we deduce by \eqref{herehartogs} that $\text{H}^j(\cM^s, \Omega_{_{\cM}}) = \text{H}^{j-1}(X, \cO_X)$ for $j = 0,1,2$. 

%To go to the whole $\cM$ as before, we use Hartogs and the genus assumptions and complete the proof.

\end{proof}

\section{For $G$ quasi-split}
In this section, we generalise \S \ref{hecke} to the quasi-split case. We do so by quoting relevant results from \cite{bt}. The final answer does not change much from the split case since our residue fields are algebraically closed.

\subsection{From Bruhat-Tits \cite{bt}} We now proceed to generalise the above to the quasi-split case. We will use notations as in \cite[\S 4]{bt}. Thus $K$ denotes a valued commutative field of infinite cardinality. Let $G$ denote a connected quasi-split reductive group defined over $K$ and let $S \subset T$ denote its maximal split torus. Since $G$ is quasi-split, the maximal torus $T$ coincides with the centraliser of $S$. The roots of $G$ with respect to $S$ are the non-zero weights of $S$ in the adjoint representation of $G$ in its Lie algebra. A {\it radicial ray} of $G$ with respect to $S$ is an open half-line with origin $0$ in $E= \mathbb{R} \otimes X^*(S)$ containing at least one root. The set of radicial rays of $G$ with respect to $S$ is denoted $\boldsymbol{\Phi}(S,G)$ or simply ${\boldsymbol \Phi}$.  For every $a \in {\boldsymbol \Phi}$, there exists a maximal closed connected subgroup $U_a$ of $G$ normalised by $S$ and such that the characters of $S$ intervening in the adjoint representation of $S$ in $Lie(U_a)$ lie in $a$. It is defined on $K$ and is a split unipotent over $K$. It is called the {\it radicial subgroup} of $G$ associated to $a$. Further, we have $U_a \neq \{1 \}$ if and only if $ a \in \boldsymbol \Phi$. 

Let $\Phi$ denote the root system of $G$ with respect to $S$. One knows that an element $a \in {\boldsymbol \Phi}$ contains one or two elements of $\Phi$. Let $V$ denote the dual of the vector subspace generated by $\Phi$ in $E$. Let $f: \Phi \ra \mathbb{R}$ be a (quasi) concave function as in \cite[\S 4.5]{bt}.  Let ${\mathfrak{G}}_f$ denote the associated Bruhat-Tits group scheme, $\overline{\mathfrak{G}}_f^0$ denote the connected component of its closed fiber and $\overline{\mathfrak{S}}$ denote the closed fiber of the split torus. Following \cite{bt}, for this subsection, let $\overline{K}$ denote the residue field which in this work we have denote by $k$ thus far. Let $R$ denote the split unipotent radical of $\overline{\mathfrak{G}}_f^0$. Let $q:\overline{\mathfrak{G}}^0_f \to \overline{\mathfrak{G}}^0_f/R$ denote the canonical quotient map. Let us recall \cite[Corollary 4.6.12]{bt}:
\begin{enumerate}
\item[(i)] The root system of $\overline{\mathfrak{G}}_f^0/R$ with respect to the maximal $\overline{K}$-split torus $q(\overline{\mathfrak{S}})$ (identified with $\overline{\mathfrak{S}}$) is the set $\Phi_f$ of $a \in \Phi$ such that $f(a)+f(-a)=0$.
\item[(ii)] The canonical map of $\text{Aut}(A)$ in $\text{Aut}(V^*)$ induces an isomorphism of $W_f$ on the Weyl group of $\Phi_f$ (operating naturally on the subspace of $V^*$ generated by $\Phi_f$ and trivially on its orthogonal supplement).
\item[(iii)] Let us assume that $\overline{K}$ is separably closed and let $R'$ denote the unipotent radical of $\overline{\mathfrak{G}}^0_f$. The root system of the reductive group $\overline{\mathfrak{G}}^0_f/R'$ (defined on the algebraic closure of $\overline{K}$) is the set $\Phi^{nm}_f$ of non-multipliable elements of $\Phi_f$.
\end{enumerate}

Denoting by $q' :\overline{\mathfrak{G}}^0_f \to \overline{\mathfrak{G}}^0_f/R'$ the canonical quotient map, $q'(\overline{\mathfrak{S}})$ is a $\overline{K}$-maximal torus of $\overline{\mathfrak{G}}^0_f/R'$.

{\em In our situation, $K$ is strictly Henselian, and  $\overline{K}$ is in fact algebraically closed. Whence, $R = R'$, and we do not encounter quasi-reductivity, and all the conditions for the above theorem are fulfilled and usable. In fact, as in our case, since the root system $\Phi$ is {\em reduced}  we have a coincidence $\Phi^{nm}_f = \Phi_f$ }. 

%\begin{thm} \cite[Th\'eor\`eme 4.6.33]{bt} Let us suppose that the valuation $\omega$ is discrete. Let $F$ be a facet of $\mathfrak{I}$ and let $\mathfrak{I}(F)$ be the star of $F$ in the polysimplicial complex $\mathfrak{I}$, i.e. the set of facets $F'$ of $\mathfrak{I}$ such that $F \subset \overline{F'}$, ordered by the relation $F' \subset \overline{F}''$. 
%\begin{enumerate}
%\item[(i)] The map $F' \mapsto p(F')= Im \overline{i}_{F,F'}^0$ is an isomorphism of ordered sets of $\mathfrak{I}(F)$ onto the set of $\overline{K}$-pseudo-pararoblic subgroups of $\overline{\mathfrak{G}}_{F}^0$ ordered by the relation opposite to the inclusion. When $\overline{K}$ is perfect, the latter are exactly the parabolic subgroups. 
%\item[(ii)] $p(F')(\overline{K})$ is the product of the canonical image of $\mathfrak{G}^0_{F'}(\mathcal{O})$ by $\overline{\mathfrak{I}}^0(\overline{K})$.
%\item[(iii)] The inverse image of $p(F')(\overline{K})$ in $\mathfrak{G}_F^0(\mathcal{O})$ is equal to $\mathfrak{G}_{F'}^0(\mathcal{O})$. 
%\end{enumerate}
%\end{thm}

 We now recall some foundational material from loc.cit pages 134-138. Let $\cI(\tt F)$ denote the {\em star of $\tt F$} in the simplicial complex $\cI$, i.e., the set of facets $\tt F'$ of $\cI$ such that $\tt F \subset \bar{\tt F}'$ and ordered by the relation $\tt F' \leq \tt F''$ if $\tt F' \subset \bar{\tt F}''$.

Let $\cG_{\tt F}$, and $\cG_{\tt F'}$ be the associated Bruhat-Tits group schemes (which we assume to be connected) on $\spec \cO$ with generic fibre $\cG_K$. Then, the identity map on $\cG_K$ extends to a morphism:
\begin{equation}
i_{_{\tt F,\tt F'}}:\cG_{\tt F'} \to \cG_{\tt F}.
\end{equation}
Let $\bar{i}_{_{_{{\tt F},{\tt F}'}}}$ denote the induced morphism:
\begin{equation}
\bar{i}_{_{_{{\tt F},{\tt F}'}}}:\cG_{_{{\tt F'},k}}  \to \cG_{_{{\tt F},k}}
\end{equation}
over the residue field $k$ (which for us is algebraically closed). By \cite[4.6.33]{bt}, the map:
\begin{equation}
F' \mapsto \text{Im} \left(\bar{i}_{_{_{{\tt F},{\tt F}'}}} \right)
\end{equation}
is an isomorphism of the ordered subsets of the star $\cI({\tt F})$ of ${\tt F}$ onto the set of parabolic subgroups of the closed fibre $\cG_{_{_{{\tt F},k}}}$. 

%For an independent self-contained proof of this fact in the generically split case, see \cite[Cor 4.3.12]{pp}.

\subsection{Generalisation of \S \ref{hecke} to quasi-split case} This last statement implies two facts which are immediate and are basic for our perspective.
Namely, one, the inverse image of $\text{Im}\left(\bar{i}_{_{_{{\tt F},{\tt F}'}}} \right)$ in $\cG_{\tt F}(\cO)$ is equal to $\cG_{\tt F'}(\cO)$.  Two, let 
\begin{equation}
P_{_{_{{\tt F},{\tt F}'}}}:= \text{Im} \left(\bar{i}_{_{_{{\tt F},{\tt F}'}}} \right)
\end{equation} 
be the parabolic subgroup of $\cG_{_{{\tt F},k}}$ determined by ${\tt F}' \in \cI(F)$. In other words, the subspace of sections in $\cG_{\tt F}(\cO)$ whose canonical image under the residue map lies in the parabolic $P_{_{_{{\tt F},{\tt F}'}}}$ is precisely $\cG_{{\tt F}'}(\cO)$. We now apply these group theoretical facts to our parahoric-correspondence setting as follows.

 Let $\tt F_{_j} \subset \bar{\tt F'}$, $j = 1,2$, i.e., the facet $\tt F'$ lies in the star $\cI(\tt F_{_j})$ for $j =1,2$. Then we have a diagram of morphism of affine group schemes over $\spec \cO$:
\begin{equation}\label{heckeatgplevel}
\begin{tikzcd}
	& {\cG_{\tt F'}} \\
	{\cG_{_{\tt F_1}}} && {\cG_{_{\tt F_2}}}
	\arrow["{i_{_{\tt F_1,\tt F'}}}"', from=1-2, to=2-1]
	\arrow["{i_{_{\tt F_2,\tt F'}}}", from=1-2, to=2-3]
\end{tikzcd}\end{equation} 
These morphisms give rise to the Parahoric-correspondence morphisms $\pi$ and $h$ in \eqref{hecke}. Let us describe the fibers of these maps.

Given a pair $(\tt F, \tt F')$ such that $\tt F \subset \bar{\tt F}'$, we get the canonically induced morphism:
\begin{equation}\label{theflagfibres}
\mf_{_{_{{\tt F},{\tt F}'}}}:\cF\ell_{\tt F'} \to \cF\ell_{\tt F}
\end{equation}
which is a \'etale locally trivial fibration with fibre type isomorphic to the homogeneous space $${\cG_{_{{\tt F},k}}/ P_{_{_{{\tt F},{\tt F}'}}}}.$$ This map descends to the morphism $\pi_{_{_{{\tt F},{\tt F}'}}}$ of the corresponding moduli stacks which continue to remain \'etale locally trivial fibrations with same fibres. Thus the Parahoric-correspondence diagram \eqref{heckecorresp} generalises to the quasi-split setup and moreover in the context of \eqref{heckeatgplevel}, we have maps $\pi_{_{_{{\tt F_j},{\tt F}'}}}$,  with fibres $${\cG_{_{{\tt F_j},k}}/ P_{_{_{{\tt F_j},{\tt F}'}}}} \quad \text{for} \quad j = 1,2.$$

Further, the relations \eqref{ModStackasE(G/P)} realising moduli stacks as fibrations obtained from the universal torsors continues to holds.

\subsubsection{The root system} Let $\Phi_{_{\tt F}}^+=\Phi^+ \cap \Phi_{_{\tt F}}$ and $\tt B$ be the basis of $\Phi_{_{\tt F}}$ corresponding to the positive root system of $\Phi_{_{\tt F}}^+$. For $\tt J \subset B$, we denote by $\Phi_{_{\tt J}}$ the set of $a \in \Phi_{_{\tt F}}$ which are the linear combinations of elements of $\tt J$ and $\tt F_J$ the facet of $\cI$ characterised by the relation ${\tt F} \subset \overline{\tt F}_J \subset \overline{C}$ and $$\tt J = \left\{a \in B \mid f_{_{\tt F_J}}(a)=f_{_{\tt F}}(a) \right\}.$$

The map $\tt J \mapsto {\tt F}_J$ is a bijection of $2^{\tt B}$  onto the set of facets $\tt F'$ such that ${\tt F} \subset \overline{\tt F}' \subset \overline{C}$, and we have $\Phi_{_{\tt F_J}}= \Phi_{_{\tt J}}$.  

We have $f_{_{\tt F}} \leq f_{_{{\tt F_J}}}$ and $f_{_{{\tt F_J}}}$ is less than the optimis\'ee of $f^*_{_{\tt F}}$. It follows from \cite[4.6.10]{bt} that the parabolic $P_{_{_{{\tt F},{\tt F_J}}}} = \text{Im} \left(\bar{i}_{_{_{{\tt F},{\tt F_J}}}} \right)$  contains the unipotent radical $R$ of $\overline{\mathfrak{G}}_{\tt F}^0$. 

Since $G$ is simply-connected, the subgroup scheme $\mathfrak T$ of ${\mathfrak{G}}^0_{\tt F}$ extending $T_K$ is {\em smooth}. We let $\overline{\mathfrak T^0}$ denote the connected component of its closed fibre.  

The parabolic $P_{_{_{{\tt F},{\tt F_J}}}}$ also contains $\overline{\mathfrak T^0}$ and, in view of \cite[4.6.5]{bt}, its image in the reductive quotient $\overline{\mathfrak{G}}^0_{\tt F}/R$ is generated by the image of $\overline{\mathfrak T^0}$ and the radicial subgroups of $\overline{\mathfrak{G}}_{\tt F}^0/R$ corresponding to $\Phi^+_{\tt F} \cup \Phi_J$.

\subsubsection{Back to the basic case} After the generalities, we now specialize to our situation. Recall the set $\mu(\alpha)$ \eqref{mu}.

When $\tt F=0$, $\Phi_{_{\tt F}}=\Phi$. We set $\tt F'$ as the open line segment joining $0$ with $\theta_{\alpha}$. Thus, we have ${\tt J} = S - \alpha$ and hence the equation \eqref{hside} continues to hold in the quasi-split case. When $\tt F=\theta_{\alpha}$. Then $f_{_{\tt F}} = f_{_{\theta_\alpha}}$ \eqref{fthetaalpha} and $f_{_{\tt F'}}=f_{_{\tt F_j}}= m_{_{r,\alpha}}$ \eqref{mralpha} above. 

Thus, for a root $r \in \Phi$, we have $$f_{_{\tt F}}(r) < f_{_{\tt F_j}}(r)\iff r \in \mu(\alpha) .$$ Hence, in the quasi-split case, \eqref{piside} continues to hold as well. Consequently, when $c_{\alpha}=1$, we get \eqref{piversush} in the quasi-split case. We isolate the following:
\begin{thm}
The equations \eqref{hside} and \eqref{piversush} continue to hold in the quasi-split case. As a consequence we get \eqref{key1}. As to \eqref{key2}, it clearly holds in this general situation as well since it was a computation at the closed fibre. 
\end{thm}

\subsection{Computations for the parahoric moduli in the tamely split case}\label{quasisplitcase} Let $\cH$ be a group scheme which is a parahoric group scheme,  of type $G$.  We assume that we have a tamely ramified Galois cover $\phi:Y \to X$ with Galois group $\Gamma$, and a split semisimple group scheme on $Y$ such that $$\cH := \phi_{_{*}}^{^{\Gamma}} \left(G_Y \right).$$ Thus, $\cH_{_K}$ is quasi-split and we can apply the last results.

Let $\cM^{^s}_{_{\cH}}$ be the stack of {\em stable} $\cH$-torsors on $X$, stability being defined in an equivariant fashion following \cite{bs}. Let $\hfr := \text{Lie}(\cH)$ and let $E_{_{\cH}}$ be the universal torsor on $X \times \cM^{^s}_{_{\cH}}$. 
\begin{defi}
Let $\cT_{_{\cH}} := p_{_*} \left(E_{_{\cH}}(\hfr) \right)$.
\end{defi}
We now indicate why the Theorems (\ref{main1}) and (\ref{main2}) hold for the stack $\cM^{^s}_{_{\cH}}$ with coefficients in $\cT_{_{\cH}}$ and $\Omega_{_{\cH}} := \cT_{_{\cH}}^*$. We first choose a base point $x \in X$ away from the ramification locus of the group scheme $\cH$. The closed fibre of $\cH$ at $x$ being $G$, we set up the Parahoric-correspondence to get a diagram \eqref{hecke} with the modification at the point $x$ alone. All the arguments now work with $E(\gfr)$ being replaced by $E_{_{\cH}}(\hfr)$.

\bibliographystyle{alpha}
\bibliography{DefBunG}
%\end{thebibliography}

\end{document}